\documentclass[12pt]{article}
\usepackage{amsmath, amsthm, amscd, amsfonts, amssymb, graphicx, color, epstopdf}
\usepackage[bookmarksnumbered, plainpages]{hyperref}
\usepackage{caption}
\usepackage{subcaption}
\usepackage{amsfonts}
\usepackage{amssymb}
\usepackage{graphics}
\usepackage{enumerate}
\usepackage{amsmath, amsthm, latexsym}
\theoremstyle{plain}
\newtheorem{Th}{Theorem}[section]
\newtheorem{Lem}[Th]{Lemma}
\newtheorem{Prop}[Th]{Proposition}

\theoremstyle{definition}
\newtheorem{Ex}[Th]{Example}

\newtheorem{Def}[Th]{Definition}

 \date{ }

\begin{document}
\title{Some results on core EP Drazin matrices and partial isometries}
\author{Gholamreza Aghamollaei\footnote{Corresponding author. Department of Pure Mathematics,
Faculty of Mathematics and Computer, Shahid Bahonar University of Kerman, Kerman, Iran. Email: aghamollaei@uk.ac.ir, aghamollaei1976@gmail.com}
\footnote{Mahani Math Center, Afzalipour Research Institute, Shahid Bahonar University of Kerman, Kerman, Iran.}, Mahdiyeh Mortezaei\footnote{Department
of Pure Mathematics, Faculty of Mathematics and Computer, Shahid Bahonar University of Kerman,
Kerman, Iran. Email: mahdiyehmortezaei@math.uk.ac.ir, mahdiyehmortezaei@gmail.com}, Dijana Mosi\'c\footnote{Faculty of Sciences and Mathematics, University of Ni\v s, Ni\v s, Serbia. Email: dijana@pmf.ni.ac.rs} \\ and N\'estor Thome\footnote{Instituto Universitario de Matem\'atica Multidisciplinar,
Universitat Polit\`ecnica de Val\`encia, Val\`encia 46022, Spain. Email: njthome@mat.upv.es}}

\maketitle{} {\footnotesize {\bf Abstract.}
}In this paper, by using the core EP inverse and the Drazin inverse which are two well known generalized inverses,
a new class of matrices entitled \textit{core EP Drazin matrices} (shortly, $CEPD$ matrices) is introduced.
This class contains the set of all EP matrices and also the set of normal matrices. Some algebraic properties of these matrices are also
investigated. Moreover, some results about the Drazin inverse and the core EP inverse of partial isometries are derived, and using them,
some conditions for which partial isometries are CEPD, are obtained. To illustrate the main results, some numerical examples are given.



 \noindent
\section{Introduction}
Let $ \mathbb{M}_{m \times n} (\mathbb{C})$ be the vector space of $ m \times n $ matrices over the field $ \mathbb{C} $ of complex numbers. The algebra $\mathbb{M}_n(\mathbb{C})$ is used instead of $\mathbb{M}_{n \times n}(\mathbb{C})$.

 For $ A\in\mathbb{M}_{m \times n}(\mathbb{C})$, the Moore-Penrose inverse of $A$ is a matrix
$ X\in\mathbb{M}_{n \times m}(\mathbb{C})$ that satisfies the following properties:
\begin{equation}\label{MP}
X A X = X, \ \ \ A X A = A, \ \ \
(A X)^\ast = A X, \ \ \
(X A)^\ast = X A,
\end{equation}
where $^{\ast}$ denotes the conjugate transpose of a matrix.
Such matrix $X$ always exists, is unique and is denoted by $A^\dagger$ \cite{P}.

 Next, for a square matrix $ A\in\mathbb{M}_n(\mathbb{C})$ with index $k$,
we recall some other known generalized inverses of $A$. It should be highlighted that the index of $A$, represented as ${ \rm ind}(A)$, refers to the smallest non-negative integer $k$ for which ${\rm rank}(A^k) = {\rm rank}(A^{k+1})$.\\
$\bullet$ The Drazin inverse of $A$ \cite{Dr} is a matrix $X\in\mathbb{M}_n(\mathbb{C})$ which is characterized by the following properties:
\begin{equation}\label{D}
A X = X A, \ \ \ X A X=X, \ \ \
X A^{k+1} = A^k. \\
\end{equation}
This matrix $X$ always exists, is unique and is denoted by $A^D$. If ${\rm ind}(A)\leq 1 $, then the Drazin inverse of $A$ is
called the group inverse of $A$ and is denoted by $A^\#$. \\
$\bullet$ The DMP (Drazin Moore-Penrose) inverse of $A$ \cite{MT} is the unique matrix $ X:= A^{D,\dagger} \in\mathbb{M}_n(\mathbb{C})$
which satisfies
\begin{center}
$X A = A^D A,\ \ \ \ \ X A X = X,\ \ \ \ \
A^k X = A^k A^\dagger.$
\end{center}
$\bullet$ The MPD (Moore-Penrose Drazin) inverse of $A$ \cite{MT} is the unique matrix $ X:=A^{\dagger,D} \in\mathbb{M}_n(\mathbb{C})$ which satisfies
\begin{center}
$A X = A A^D, \ \ \ \ \ X A X=X, \ \ \ \ \
X A^k = A^\dagger A^k$.
\end{center}
$\bullet$ The matrix $A$ can be written as $ A = A_1 + A_2 $ (the core nilpotent decomposition of $A$) \cite[Theorem 2.2.21]{MBM},
where $A_1 = A A^D A $ is the core part of $A$ with property ${\rm ind}(A_1)\leq 1$, $ A_2 = A-A A^D A $ is the nilpotent part of $A$, and $A_1 A_2 = 0 = A_2 A_1.$ The CMP (Core Moore-Penrose) inverse
of $A$ \cite{MS} is the unique matrix $ X:=A^{c\dagger} \in\mathbb{M}_n(\mathbb{C})$ that satisfies the following properties:
\begin{center}
$A X A = A_1, \ \ X A X = X, \ \
X A = A^\dagger A_1, \ \ AX = A_1 A^\dagger$.
\end{center}
$\bullet$ The core EP inverse of $A$ \cite{PM} is the unique matrix $ X:=A^{\scriptsize\textcircled{$\dagger$}} \in\mathbb{M}_n(\mathbb{C})$ which satisfies the following equations:
\begin{equation}\label{CEP}
X A X = X,\ \ \ \ \
\mathcal{R}(X) = \mathcal{R}(A^k) = \mathcal{R}(X^\ast),
\end{equation}
where $\mathcal{R}(.) $ is the range (column) space of a matrix.

 Let $ A\in\mathbb{M}_n(\mathbb{C})$. The matrix $A$ is EP (Equal Projection) if $A A^\dagger = A^\dagger A $. In 2016, Malik, Rueda and Thome in \cite{MRT} introduced $k-$EP matrices, and in 2018, Mehdipour and Salemi in \cite{MS} introduced core EP matrices. Note that
$A$ is $k-$EP, where $k = {\rm ind}(A)$, if $ A^k A^\dagger = A^\dagger A^k $, and it is called a
core EP matrix if $A^\dagger A_{1} = A_{1} A^\dagger$, where $A_1$ is the core part of $A$. Moreover,
$A$ is called SD (star-dagger) if $A^\ast A^\dagger = A^\dagger A^\ast $ \cite{HS}, and
it is said to be a partial isometry if $A^\dagger = A^\ast $, or equivalently, $A A^\ast A = A$ \cite{HSS}.

Generalized inverses have many applications in various areas; see for example \cite{BG, CM, CMR, IW, MBM}. One of the important generalized inverses is the Drazin inverse.
It has numerous applications in areas such as multibody system dynamics \cite{SFR},
distinguishability of linear descriptor control systems \cite{DY}, and the theory of finite Markov chains \cite{CM,ZCh}.
In addition, the Drazin inverse is used in linear systems of differential equations with constant coefficients, which frequently occur in electrical circuits and are very important in electronic engineering \cite[Chapters 5 and 9]{CM}.
Another important type of generalized inverse is the core EP inverse.
It has many applications in areas such as
constrained matrix approximation problems \cite{MSK}, semistable matrices \cite{GCP}, and solving certain systems of linear equations
\cite{JW,MSM}.
Considering these applications, one of the aims of this paper is to introduce and study a new class of matrices,
referred to as \textit{core EP Drazin matrices}. Another goal of this paper is to focus on partial isometries and some of their
generalized inverses. Note that partial isometries have many applications in various fields, such as problems of best approximation \cite{ACh}, medical imaging (e.g., MRI) \cite{AD}, quantum holonomy \cite{U}, and the polar decomposition, which has numerous applications in areas such as light scattering in particle systems \cite{SAM}.

This paper is organized as follows. In Section $2$, we present some preliminaries and essential properties of the core EP inverse, DMP inverse, MPD inverse, and CMP inverse, which we frequently use in our investigation.
In Section $3$, by utilizing the core EP inverse and the Drazin inverse, we introduce core EP Drazin (CEPD) matrices. In this section,
we also investigate several algebraic properties of these matrices. In Section $4$, we consider
partial isometries, and we find a necessary and sufficient condition under which the Drazin
inverse, core EP inverse, DMP inverse, and MPD inverse of a partial isometry are also partially isometric.
Furthermore, we present some results for partial isometries when they are CEPD.
In Section $5$, we discuss applications of CEPD matrices and partial isometries in solving systems of linear equations.
To illustrate the main results of the paper, we
provide numerical examples.


\noindent
\section{Preliminaries}

 In this section, we present some lemmas about the Drazin inverse, core EP inverse, MPD inverse, DMP inverse, and CMP inverse which are useful in our discussion.

\begin{Lem}\label{DMP 1}
For $ A\in\mathbb{M}_n(\mathbb{C})$ with the core part $A_1$, the following assertions are true:
\begin{enumerate}[(i)]
\item (\cite[Theorem 2.2]{MT}) $ A^{D,\dagger} = A^D A A^\dagger $;
\item (\cite{MT}) $ A^{\dagger,D} = A^\dagger A A^D$;
\item (\cite[Proposition 2.2(1)]{MS}) $ A^\dagger A A^D A A^\dagger = A^{c\dagger}$;
\item (\cite[Theorem 3.2]{KDY}) $ A^{D,\dagger} A^D = (A^D)^2$;
\item (\cite[Theorem 3.5]{KDY}) $ A^{D,\dagger} A = A A^D $;
\item (\cite[Ex. 27 in p. 166]{BG}) $(A^D)^\# = A^2 A^D = A_1$.
\end{enumerate}
\end{Lem}

\begin{Lem}\label{CEP 1}
For $ A\in\mathbb{M}_n(\mathbb{C})$ with $ind(A) = k $,
the following assertions are true:
\begin{enumerate}[(i)]
\item (\cite[Lemma 2.3]{FLT1}) $A^{\scriptsize\textcircled{$\dagger$}} A^{k+1} = A^k$;
\item (\cite[Lemma 2.3]{FLT1}) $ A A^{\scriptsize\textcircled{$\dagger$}}$ is Hermitian;
\item (\cite[Lemma 2.6]{ZC}) $ A^{\scriptsize\textcircled{$\dagger$}} = A^D A^m (A^m)^\dagger $ for any positive integer $m\geq k$;
\item (\cite[Lemma 2.6]{ZC}) $A^{\scriptsize\textcircled{$\dagger$}} = A^k (A^{k+1})^\dagger$;
\item (\cite[Theorem 3.2]{FLT}) $A^{\scriptsize\textcircled{$\dagger$}} A^k = A^D A^k$;
\item (\cite[Theorem 2.7]{FLT1}) $A^{\scriptsize\textcircled{$\dagger$}} =(A^{k+1} (A^k)^\dagger)^\dagger $;
\item (\cite[Theorem 2.9]{FLT1}) $A^{\scriptsize\textcircled{$\dagger$}} $ is $EP$;
\item (\cite[Theorem 2.9]{FLT1}) $ A (A^{\scriptsize\textcircled{$\dagger$}})^2 = A^{\scriptsize\textcircled{$\dagger$}}$;
\item (\cite[Corollary 3.4]{W}) $ A A^{\scriptsize\textcircled{$\dagger$}} = A^k (A^k)^\dagger $;
\item (\cite[Theorem 4.1]{ZCH}) $ A^D A^{\scriptsize\textcircled{$\dagger$}} = (A^{\scriptsize\textcircled{$\dagger$}})^2 $.
\end{enumerate}
\end{Lem}

For $ A\in\mathbb{M}_n(\mathbb{C})$ with ${\rm rank}(A) = r > 0 $, the Hartwing-Spindelb\"{o}ck decomposition of $A$ \cite{HS} is given by:
\begin{equation}\label{HS}
A = U \left [
\begin{array} {cc}
\Sigma K & \Sigma L \\
0 & 0
\end{array} \right ] U^\ast,
\end{equation}
where $ U\in \mathbb{M}_n(\mathbb{C}) $ is unitary, $\Sigma = diag(\sigma_{1} I_{n_{1}}, \sigma_{2} I_{n_{2}}, \ldots , \sigma_{l} I_{n_{l}}) $ with $ \sigma_{1} > \sigma_{2} > \cdots > \sigma_{l} > 0 $ which are the singular values
of $A$, $ n_{1} + n_{2} + \cdots + n_{l} = r$, $ K\in \mathbb{M}_r(\mathbb{C}) $ and $ L\in \mathbb{M}_{r \times (n-r)}(\mathbb{C}) $ satisfy
$ K K^\ast + L L^\ast = I_{r} $. It is known that ${\rm ind}(\Sigma K) = {\rm ind}(A)-1$ \cite[Lemma 2.8]{MT}. Now, we state the following lemma. Recall that for $ A\in\mathbb{M}_n(\mathbb{C})$, the Drazin-Star matrix of $A$ (or the Drazin-Star inverse of $(A^\dagger)^\ast$) is defined,
see e.g., \cite[Def. 2.1]{D}, as $A^{D,\ast} := A^D A A^\ast$.

\begin{Lem}\label{HSD 2}
For $ A\in\mathbb{M}_n(\mathbb{C})$ with the Hartwing-Spindelb\"{o}ck decomposition as in (\ref{HS}),
the following assertions are true:
\begin{enumerate}[(i)]
\item (\cite[Relation (14)]{MT}) $ A^D = U \begin{bmatrix}
(\Sigma K)^D & ((\Sigma K)^D)^2 \Sigma L \\
0 & 0 \\
\end{bmatrix} U^\ast $;
\item (\cite[Theorem 3.2.]{FLT1}) $ A^{\scriptsize\textcircled{$\dagger$}} = U \begin{bmatrix}
(\Sigma K)^{\scriptsize\textcircled{$\dagger$}} & 0 \\
0 & 0 \\
\end{bmatrix} U^\ast $;
\item (\cite[Theorem 2.5]{MT}) $A^{D,\dagger} = U \begin{bmatrix}
(\Sigma K)^D & 0 \\
0 & 0 \\
\end{bmatrix} U^\ast $;
\item (\cite[Theorem 2.7]{D}) $A^{D,\ast} = U \begin{bmatrix}
(\Sigma K)^D \Sigma \Sigma^\ast & 0 \\
0 & 0 \\
\end{bmatrix} U^\ast $.
\end{enumerate}
\end{Lem}


\noindent
\section{Core EP Drazin matrices}

 At the beginning of this section, by using the Drazin inverse and the core EP inverse, we introduce core EP Drazin matrices as follows:

 \begin{Def}\label{CEPD}
A matrix $ A\in\mathbb{M}_n(\mathbb{C})$ is called \textit{core EP Drazin} (shortly, $CEPD$) matrix if
$A^{\scriptsize\textcircled{$\dagger$}} A^D = A^D A^{\scriptsize\textcircled{$\dagger$}}$.
\end{Def}

 In the following theorem, we give some equivalent conditions for CEPD matrices. For this, we recall that
the Drazin inverse $A^D$ double commutes with $A$ (i.e., $AB=BA$ implies that $A^DB=BA^D$) \cite{BG}.

 \begin{Th}\label{CEPD equivalent}
For $ A\in\mathbb{M}_n(\mathbb{C})$ with $ind(A) = k$ and the core part $A_1$, the following statements are equivalent:
\begin{enumerate}[(i)]
\item $A$ is CEPD;
\item $A_1 A^{\scriptsize\textcircled{$\dagger$}} = A^{\scriptsize\textcircled{$\dagger$}} A_1$;
\item $ A^DA^{k+1} (A^k)^\dagger=A^{k+1} (A^k)^\dagger A^D$;
\item $ A_1A^{k+1} (A^k)^\dagger=A^{k+1} (A^k)^\dagger A_1$;
\item $A^{k+2} (A^k)^\dagger = A A_1$;
\item $A^{\scriptsize\textcircled{$\dagger$}}= A^D$;
\item $(A^{\scriptsize\textcircled{$\dagger$}})^m=(A^D)^m$ for any positive integer $m\geq 2$;
\item $(A^{\scriptsize\textcircled{$\dagger$}})^m=(A^D)^m$ for some positive integer $m\geq 2$.
\end{enumerate}
\end{Th}

 \begin{proof} Assume that $A$ is CEPD, i.e., $A^D A^{\scriptsize\textcircled{$\dagger$}} = A^{\scriptsize\textcircled{$\dagger$}} A^D$. It is known,
by Lemma \ref{DMP 1}$(vi)$, that $(A^D)^\#=A_1$. Now, by the fact that the group inverse $(A^D)^\#$ double commutes with $A^D$
(i.e., $(A^D)^\# A^{\scriptsize\textcircled{$\dagger$}} = A^{\scriptsize\textcircled{$\dagger$}} (A^D)^\#$), we conclude that
$A_1 A^{\scriptsize\textcircled{$\dagger$}} = A^{\scriptsize\textcircled{$\dagger$}} A_1 $. \\
If $ A_1 A^{\scriptsize\textcircled{$\dagger$}} = A^{\scriptsize\textcircled{$\dagger$}} A_1 $, by $A_1^\#=A^D$ and double commutativity
(i.e., $ A_1^\# A^{\scriptsize\textcircled{$\dagger$}} = A^{\scriptsize\textcircled{$\dagger$}} A_1^\# $),
we deduce that $A^D A^{\scriptsize\textcircled{$\dagger$}} = A^{\scriptsize\textcircled{$\dagger$}} A^D$. So, $(i)$ is equivalent to $(ii)$.

 By Lemma \ref{CEP 1}$(vi)$, we have
$A^{\scriptsize\textcircled{$\dagger$}} =(A^{k+1} (A^k)^\dagger)^\dagger$.
Also, we know, by Lemma \ref{CEP 1}$(vii)$, that $A^{\scriptsize\textcircled{$\dagger$}}$ is $EP$. Therefore,
$(A^{\scriptsize\textcircled{$\dagger$}})^\dagger = (A^{\scriptsize\textcircled{$\dagger$}})^\#$, and so,
$((A^{k+1} (A^k)^\dagger)^\dagger)^\dagger = ((A^{k+1} (A^k)^\dagger)^\dagger)^\#$. It follows that
$(A^{\scriptsize\textcircled{$\dagger$}})^\# = A^{k+1} (A^k)^\dagger $. So, if
$A^D A^{\scriptsize\textcircled{$\dagger$}} = A^{\scriptsize\textcircled{$\dagger$}} A^D$, by double commutativity
(i.e., $A^D (A^{\scriptsize\textcircled{$\dagger$}})^\# = (A^{\scriptsize\textcircled{$\dagger$}})^\# A^D$), we see that
$ A^DA^{k+1} (A^k)^\dagger=A^{k+1} (A^k)^\dagger A^D$; as required. \\
If $ A^D A^{k+1} (A^k)^\dagger=A^{k+1} (A^k)^\dagger A^D$, by the fact that $(A^{k+1} (A^k)^\dagger)^\# = A^{\scriptsize\textcircled{$\dagger$}}$
and double commutativity
(i.e., $ A^D (A^{k+1} (A^k)^\dagger)^\# = (A^{k+1} (A^k)^\dagger)^\# A^D$),
we have $A^D A^{\scriptsize\textcircled{$\dagger$}} = A^{\scriptsize\textcircled{$\dagger$}} A^D$.
So, $(i)$ is equivalent to $(iii)$.

 Similarly, by $(A^D)^\# = A_1$ and $A_1^\#=A^D$, and also, by double commutativity, we have that $(iii)$ is equivalent to $(iv)$.

 To prove $(iv)\Leftrightarrow (v)$, by the fact $A_1 = A A^D A$ and using (\ref{D}), we have
\begin{eqnarray*}
A_1 A^{k+1} (A^k)^\dagger = A^{k+1} (A^k)^\dagger A_1 & \Leftrightarrow & A A^D A A^{k+1} (A^k)^\dagger = A^{k+1} (A^k)^\dagger A A^D A \\
& \Leftrightarrow & A^2 A^D A^{k+1} (A^k)^\dagger = A^{k+1} (A^k)^\dagger A^k (A^D)^k A \\
& \Leftrightarrow & A^2 A^k (A^k)^\dagger = A A^k (A^D)^k A \\
& \Leftrightarrow & A^{k+2} (A^k)^\dagger = A A_1;
\end{eqnarray*}
as required.

 By Lemma \ref{CEP 1}($(iii)$ and $(x)$), we have
$A^{\scriptsize\textcircled{$\dagger$}} A^D = A^D A^{\scriptsize\textcircled{$\dagger$}}$ if and only if
$A^D A^k(A^k)^\dagger A^D = (A^{\scriptsize\textcircled{$\dagger$}})^2$. From
\begin{eqnarray*}
A^D A^k(A^k)^\dagger A^D&=&A^D A^k(A^k)^\dagger A^k(A^D)^{k+1}\\&=&A^D A^k(A^D)^{k+1}\\&=&(A^D)^2,
\end{eqnarray*}
we get $(A^D)^2=(A^{\scriptsize\textcircled{$\dagger$}})^2$, and so, by Lemma \ref{CEP 1}$(viii)$, we see that
$$A^D=A(A^D)^2=A(A^{\scriptsize\textcircled{$\dagger$}})^2=A^{\scriptsize\textcircled{$\dagger$}}.$$ Therefore, $(i)$ implies $(vi)$, and so, $(i)$ and $(vi)$ are equivalent.

 To complete the proof, it is enough to prove that $(viii)$ implies $(vi)$. For this,
suppose that $(A^{\scriptsize\textcircled{$\dagger$}})^m=(A^D)^m$ for some $m\geq 2$. Then
by Lemma \ref{CEP 1}$(viii)$, we see that
$A^D=A^{m-1}(A^D)^m=A^{m-1}(A^{\scriptsize\textcircled{$\dagger$}})^m=A^{\scriptsize\textcircled{$\dagger$}}$.
Hence, the result holds.
\end{proof}

 Using the Hartwing-Spindelb\"{o}ck decomposition as in (\ref{HS}), we state the following theorem.

 \begin{Th}\label{HSD}
For $ A\in\mathbb{M}_n(\mathbb{C})$ with the Hartwing-Spindelb\"{o}ck decomposition as in (\ref{HS}) and $ind(A) = k$, the following statements are equivalent:
\begin{enumerate}[(i)]
\item $A$ is CEPD;
\item $\Sigma K$ is CEPD and $ (\Sigma K)^{k-1} \Sigma L = 0 $;
\item $ A^{\scriptsize\textcircled{$\dagger$}} = A^{D,\dagger} $ and $(\Sigma K)^{k-1} \Sigma L = 0$.
\end{enumerate}
\end{Th}

 \begin{proof}
By Lemma \ref{HSD 2}($(i)$ and $(ii)$) and using Theorem \ref{CEPD equivalent}$((i) \Leftrightarrow (vi))$, we conclude that
$A$ is CEPD if and only if $\Sigma K$ is CEPD and $((\Sigma K)^D)^2\Sigma L = 0$. We know, by \cite[Lemma 2.8]{MT},
that ${\rm ind}(\Sigma K) = k-1$, and hence, $(\Sigma K)^{k} (\Sigma K)^D =(\Sigma K)^{k-1}$. Also, we have
$\Sigma K ((\Sigma K)^D)^2 = (\Sigma K)^D$, and hence, $((\Sigma K)^D)^2\Sigma L = 0$ if and only if
$(\Sigma K)^D \Sigma L = 0$. So, if $(\Sigma K)^D \Sigma L = 0$, then $(\Sigma K)^{k-1} \Sigma L = 0$.
Conversely, if $(\Sigma K)^{k-1} \Sigma L = 0$, then by the fact that $(\Sigma K)^{k-1} ((\Sigma K)^{2k-1})^{\dagger} (\Sigma K)^{k-1} =
(\Sigma K)^D$, we see that $(\Sigma K)^D \Sigma L = 0$. Consequently, $(i)$ is equivalent to $(ii)$.

 By Lemma \ref{HSD 2}($(ii)$ and $(iii)$) and Theorem \ref{CEPD equivalent}$((i) \Leftrightarrow (vi))$, we see that
$A^{\scriptsize\textcircled{$\dagger$}} = A^{D,\dagger}$
if and only if $\Sigma K$ is CEPD. So, $(ii)$ is equivalent to $(iii)$, and hence, the proof is complete.
\end{proof}

The following result is related to the direct sum of two CEPD matrices.

 \begin{Th}\label{D.S}
For $ A \in\mathbb{M}_n(\mathbb{C})$ and $ B \in\mathbb{M}_m(\mathbb{C})$, the matrix
$A\oplus B$ is a CEPD if and only if $A$ and $B$ are CEPD.
\end{Th}

 \begin{proof}
By \cite[ Theorem 8 in p. 164]{BG}, it is obvious that $(A\oplus B)^D = A^D \oplus B^D$. By setting $k_1 = {\rm ind}(A), \ k_2 = {\rm ind}(B), \ k=max\lbrace k_1, k_2 \rbrace $, and by using Lemma \ref{CEP 1}$(iii)$, we have
\begin{eqnarray*}
(A \oplus B)^{\scriptsize\textcircled{$\dagger$}} & = & (A \oplus B)^D (A \oplus B)^k ((A\oplus B)^k)^\dagger \\
& = & (A^D \oplus B^D) (A^k \oplus B^k) ((A^k)^\dagger \oplus (B^k)^\dagger) \\
& = & A^D A^k (A^k)^\dagger \oplus B^D B^k (B^k)^\dagger \\
& = & A^{\scriptsize\textcircled{$\dagger$}} \oplus B^{\scriptsize\textcircled{$\dagger$}}.
\end{eqnarray*}
Hence,
$(A\oplus B)^D (A\oplus B)^{\scriptsize\textcircled{$\dagger$}} = (A\oplus B)^{\scriptsize\textcircled{$\dagger$}} (A\oplus B)^D$
if and only if $ A^D A^{\scriptsize\textcircled{$\dagger$}} = A^{\scriptsize\textcircled{$\dagger$}} A^D $ and
$B^D B^{\scriptsize\textcircled{$\dagger$}} = B^{\scriptsize\textcircled{$\dagger$}} B^D $;
completing the proof.
\end{proof}

In the following result, we find those $k-$EP matrices that are CEPD.

 \begin{Th}\label{CEP and D}
For $k-$EP matrix $ A\in\mathbb{M}_n(\mathbb{C})$ with $k = ind(A) \leq 2 $, it holds that $ A^{\scriptsize\textcircled{$\dagger$}} = A^D$, and
consequently, $A$ is CEPD.
\end{Th}

 \begin{proof}
It is clear that the result holds for $k=0$. If $k = 1$, then $A$ is EP, and so, by Lemma \ref{CEP 1}$(iii)$ and using (\ref{D}), we have
$ A^{\scriptsize\textcircled{$\dagger$}} = A^D A A^\dagger = A^D A A^D = A^D $. Hence, in this case, the result also holds. Now, let
$k=2$. Since $A$ is $2-$EP, it follows, by \cite[Proposition 2.13]{MRT}, that $A^2$ is EP. So, by Lemma \ref{CEP 1}$(iii)$ and using (\ref{D}),
we have
\begin{eqnarray*}
A^{\scriptsize\textcircled{$\dagger$}} & = & A^D A^2 (A^2)^\dagger \\
& = & A^D A A^D A^2 (A^2)^\dagger \\
& = & A (A^D)^2 A^2 (A^2)^\dagger \\
& = & A (A^2)^D A^2 (A^2)^D \\
& = & A (A^2)^D \\
& =& A (A^D)^2 \\
& = & A^D;
\end{eqnarray*}
completing the proof.
\end{proof}

Note that $k-$EP matrices $A$ with $k= {\rm ind}(A)>2$ may or may not be CEPD; as the following example shows.

 \begin{Ex}
Let $ A = \begin{bmatrix}
1 & 1 & 3 \\
5 & 2 & 6 \\
-2 & -1 & -3 \\
\end{bmatrix} $ and $ B = \begin{bmatrix}
-1 & 1 & 0 & 0 & 1 & 0 \\
1 & -1 & 0 & 0 & -1 & 0 \\
0 & 0 & 0 & 1 & -1 & 1 \\
-1 & -1 & 0 & 0 & 1 & -1 \\
1 & -1 & 0 & 0 & 1 & -1 \\
1 & -1 & 0 & 0 & 0 & 0 \\
\end{bmatrix}$. Obviously, $A^3 = 0$, ${\rm ind}(A) =3 $, and so, $A$ is $3-$EP. Moreover, by Lemma \ref{CEP 1}$(iii)$, we see that
$A^{\scriptsize\textcircled{$\dagger$}} = A^D = 0$. Hence, $A$ is a CEPD matrix.
Also, we see that ${\rm ind}(B) = 3 $ and $B$ is $3-$EP. By a simple computation, we have
\begin{center}
$ B^D = \begin{bmatrix}
0 & 0 & 0 & 0 & \frac{1}{10^{4}} & \frac{4999}{10^{4}} \\
0 & 0 & 0 & 0 & -\frac{1}{10^{4}} & -\frac{4999}{10^{4}}\\
0 & 0 & 0 & 0 & 0 & 0 \\
0 & 0 & 0 & 0 & 1 & -1 \\
\frac{1}{2} & -\frac{1}{2} & 0 & 0 & \frac{3}{10^{4}} & \frac{9996}{10^{4}} \\
\frac{1}{2} & -\frac{1}{2} & 0 & 0 & -\frac{9997}{10^{4}} & \frac{4999}{25 \times 10^{2}} \\
\end{bmatrix}$, and
\end{center}

 \begin{center}
$ B^{\scriptsize\textcircled{$\dagger$}} = \dfrac{1}{6} \begin{bmatrix}
0 & 0 & 0 & -1 & 1 & 2 \\
0 & 0 & 0 & 1 & -1 & -2 \\
0 & 0 & 0 & 0 & 0 & 0 \\
0 & 0 & 0 & 4 & 2 & -2 \\
3 & -3 & 0 & -2 & 2 & 4 \\
3 & -3 & 0 & -6 & 0 & 6 \\
\end{bmatrix}$.
\end{center}
Since $B^D \neq B^{\scriptsize\textcircled{$\dagger$}} $, it follows, by Theorem \ref{CEPD equivalent}$((i) \Leftrightarrow (vi))$, that $B$ is not CEPD.
Moreover, in view of Theorem \ref{D.S}, for any $n\in \mathbb{N}$, the matrix $A\oplus I_n $ is CEPD, and $B\oplus I_n $ is not CEPD.
So, we can find a lot of matrices which are CEPD, and a lot of matrices which are not CEPD.
\end{Ex}

In the following proposition, we find some other classes of CEPD matrices.

 \begin{Prop}\label{normal is CEPD}
Let $ A\in\mathbb{M}_n(\mathbb{C})$. If $A$ is normal (i.e., $A A^\ast = A^\ast A $) or $A A^{\scriptsize\textcircled{$\dagger$}} =
A^{\scriptsize\textcircled{$\dagger$}} A $, then $A$ is CEPD.
\end{Prop}

 \begin{proof}
At first, we assume that $A$ is normal. Then, obviously, $A$ is EP, and so, by Theorem \ref{CEP and D}, $A$ is CEPD.

 Secondly, if $A$ and $A^{\scriptsize\textcircled{$\dagger$}}$ commute, then $A^{\scriptsize\textcircled{$\dagger$}}$ and $A^D$ also commute;
because $A^D$ is a polynomial in $A$ (see \cite[Lemma 5 in p. 164]{BG} or \cite[Theorem 2.4.30]{MBM}). So, the proof is complete.
\end{proof}

 We state some additional properties of CEPD matrices in the following result.

 \begin{Th}\label{properties of CEPD}
For CEPD matrix $ A\in\mathbb{M}_n(\mathbb{C})$ with $k = ind(A)$, the following assertions are true:
\begin{enumerate}[(i)]
\item $ A^k A^{\scriptsize\textcircled{$\dagger$}} = (A^k)^2 (A^k)^\dagger A^D $;
\item $ A^{c\dagger} A^D = A^\dagger A^{\scriptsize\textcircled{$\dagger$}} $;
\item $A^{\dagger,D} A^{\scriptsize\textcircled{$\dagger$}} A^\dagger = A^\dagger A^{\scriptsize\textcircled{$\dagger$}} A^\dagger $;
\item $ A^D A^k A^\ast = A^D A^k A^{D,\ast} = A^{\scriptsize\textcircled{$\dagger$}} A^k A^\ast $;
\item $(A^k)^2 (A^k)^\dagger (A^{k+1})^\dagger = A^{\scriptsize\textcircled{$\dagger$}}$.
\end{enumerate}
\end{Th}

 \begin{proof}
To prove item $(i)$, by (\ref{D}) and Lemma \ref{CEP 1}$(iii)$, we can see that
\begin{eqnarray*}
A^k A^{\scriptsize\textcircled{$\dagger$}} & = & A^{k+1} A^D A^{\scriptsize\textcircled{$\dagger$}} \\
& = & A^{k+1} A^{\scriptsize\textcircled{$\dagger$}} A^D \\
& = & A^{k+1} A^D A^k (A^k)^\dagger A^D \\
& = & (A^k)^2 (A^k)^\dagger A^D,
\end{eqnarray*}
and this completes the proof of $(i)$.
To prove the assertion $(ii)$, by Lemma \ref{DMP 1}($(iii)$ and $(i)$), Theorems \ref{HSD} and \ref{CEPD equivalent}$(vi)$
and Lemma \ref{CEP 1}$(viii)$, we can see that
\begin{eqnarray*}
A^{c\dagger} A^D & = & A^\dagger A A^D A A^\dagger A^D \\
& = & A^\dagger A A^{D,\dagger} A^D \\
& = & A^\dagger A (A^{\scriptsize\textcircled{$\dagger$}} A^{\scriptsize\textcircled{$\dagger$}})\\
& = & A^\dagger A^{\scriptsize\textcircled{$\dagger$}}.
\end{eqnarray*}
So, the result in $(ii)$ holds.
To prove the assertion $(iii)$, by using Lemma \ref{DMP 1}$(ii)$, Theorem \ref{CEPD equivalent}$(vi)$ and Lemma \ref{CEP 1}$(viii)$, we have
\begin{eqnarray*}
A^{\dagger,D} A^{\scriptsize\textcircled{$\dagger$}} A^\dagger & = & A^\dagger A A^D A^{\scriptsize\textcircled{$\dagger$}} A^\dagger \\
& = & A^\dagger A (A^{\scriptsize\textcircled{$\dagger$}})^2 A^\dagger \\
& = & A^\dagger A^{\scriptsize\textcircled{$\dagger$}} A^\dagger.
\end{eqnarray*}
So, the proof of $(iii)$ is complete.
To prove the first equality of item $(iv)$, by (\ref{D}) and the fact that $ A^{D,\ast} = A^D A A^\ast$, we have
\begin{eqnarray*}
A^D A^k A^\ast & = & A^D A^D A^{k+1} A^\ast \\
& = & A^D A^k A^D A A^\ast \\
& = & A^D A^k A^{D,\ast}.
\end{eqnarray*}
To prove the second equality of this item, by using Lemma \ref{CEP 1}$(i)$ and (\ref{D}), we have
\begin{eqnarray*}
A^D A^k A^{D,\ast} & = & A^D A^{\scriptsize\textcircled{$\dagger$}} A^{k+1} A^D A A^\ast \\
& = & A^{\scriptsize\textcircled{$\dagger$}} A^D A^{k+1} A^D A A^\ast \\
& = & A^{\scriptsize\textcircled{$\dagger$}} A^D A^{k+1} A^\ast \\
& = & A^{\scriptsize\textcircled{$\dagger$}} A^k A^\ast.
\end{eqnarray*}
So, the result in $(iv)$ holds.

 Finally, to prove item $(v)$, by Lemma \ref{CEP 1}$(ix)$, Theorem \ref{CEPD equivalent}$(vi)$, relation (\ref{D}) and Lemma \ref{CEP 1}$(iv)$,
we can see that
\begin{eqnarray*}
(A^k)^2 (A^k)^\dagger (A^{k+1})^\dagger & = & A^k A^k (A^k)^\dagger (A^{k+1})^\dagger \\
& = & A^k A A^{\scriptsize\textcircled{$\dagger$}} (A^{k+1})^\dagger \\
& = & (A^{k+1} A^D) (A^{k+1})^\dagger \\
& = & A^k (A^{k+1})^\dagger \\
& = & A^{\scriptsize\textcircled{$\dagger$}};
\end{eqnarray*}
as required. So, the proof is complete.
\end{proof}


\noindent
\section{Partial isometries}

 Recall that a matrix $ A\in\mathbb{M}_n(\mathbb{C}) $ is a partial isometry if
$ \Vert Ax \Vert = \Vert x \Vert $ for every $ x \in (ker(A))^\perp$, or equivalently, $ A^\dagger = A^\ast $; for more information see \cite{HSS}. We know that partial isometries $A$ are SD, and so by \cite[Corollary 1]{HS}, it holds that
$ A^D A^\ast A = A A^\ast A^D $. Moreover, we state the following result.

 \begin{Th}\label{Drazin- Star} For partial isometry $ A\in\mathbb{M}_n(\mathbb{C}) $, it holds that
\begin{center}
$A A^\ast (A^{D,\ast})^2 A = A^D = A^D A^\ast A $.
\end{center}
Moreover, if $k = ind(A)$, then $ A^k A^D A^\ast A^2 = A^k $ and $ A^k A^\ast A^2 = A^{k+1}$.
\end{Th}

 \begin{proof}
To prove the left equality in the first assertion, since $A$ is a partial isometry,
it follows that $A A^\ast A = A$, and hence, by the fact that $ A A^D = A^D A $, we conclude that
\begin{eqnarray*}
A A^\ast (A^{D,\ast})^2 A & = & A A^\ast A^D A A^\ast A^D A A^\ast A \\
& = & A A^\ast A A^D A^\ast A^D A \\
& = & A A^D A^\ast A^D A \\
& = & A^D A A^\ast A A^D \\
& = & A^D;
\end{eqnarray*}
as required.

 To prove the right equality in this assertion, by the fact that $(A^D)^2 A = A^D $, we see that $A^D A^\ast A = (A^D)^2 A A^\dagger A = (A^D)^2 A = A^D.$
So, the proof of the first assertion is complete.

 For the second assertion, let $ {\rm ind}(A) = k $. So, by the right equality of the first assertion
and by using (\ref{D}), we have $
A^k A^D A^\ast A^2 = A^k A^D A
= A^{k+1} A^D = A^k$;
as required. Also, we can see that $ A^{k+1} A^D A^\ast A^2 = A^{k+1} $, and hence, by (\ref{D}), we have $ A^k A^\ast A^2 = A^{k+1} $;
completing the proof.
\end{proof}

 The partial isometry assumption in Theorem \ref{Drazin- Star} is necessary; see the following example.

 \begin{Ex}
Consider $ A = \begin{bmatrix}
2 & 0 & 0 \\
-1 & 1 & 1 \\
-1 & -1 & -1 \\
\end{bmatrix} $. Obviously, $A$ is not a partial isometry. By a simple computation, we have
\begin{center}
$ A^D = \dfrac{1}{2} \begin{bmatrix}
1 & 0 & 0 \\
-1 & 0 & 0 \\
0 & 0 & 0 \\
\end{bmatrix} $, and $ A^D A^\ast A =\begin{bmatrix}
3 & 0 & 0 \\
-3 & 0 & 0 \\
0 & 0 & 0 \\
\end{bmatrix} $,
\end{center}
which are not equal.
\end{Ex}

It is clear that for normal partial isometries $A$, it holds that $A^D$ is also partial isometry; because these matrices are EP.
But, in general, the Drazin inverse, core EP inverse, the MPD inverse, and the DMP inverse of a partial isometry
may not be partial isometry; see the
following example. Also, in this example, we find a nonnormal matrix $A$ for which $A^D$ is a partial isometry.

 \begin{Ex}
\begin{enumerate}[(i)]
\item Let $ A = \dfrac{1}{2} \begin{bmatrix}
2 & 0 & 0 \\
0 & \sqrt{3} & 0 \\
0 & 1 & 0 \\
\end{bmatrix} $. Obviously, $A$ is a partial isometry. Also, $ A^D = \dfrac{1}{3} \begin{bmatrix}
3 & 0 & 0 \\
0 & 2 \sqrt{3} & 0 \\
0 & 2 & 0 \\
\end{bmatrix},$ which is not a partial isometry; because
\begin{center}
$(A^D)^\dagger = \dfrac{1}{8} \begin{bmatrix}
8 & 0 & 0 \\
0 & 3 \sqrt{3} & 3 \\
0 & 0 & 0 \\
\end{bmatrix} $.
\end{center}
Also, we have that
\begin{center}
$ A^{D,\dagger} = A^D A A^\dagger = \dfrac{1}{6} \begin{bmatrix}
6 & 0 & 0 \\
0 & 3 \sqrt{3} & 3 \\
0 & 3 & \sqrt{3} \\
\end{bmatrix} $ and $ (A^{D,\dagger})^\dagger = \dfrac{1}{8} \begin{bmatrix}
8 & 0 & 0 \\
0 & 3 \sqrt{3} & 3 \\
0 & 3 & \sqrt{3} \\
\end{bmatrix} $.
\end{center}
Thus, $ (A^{D,\dagger})^\ast \neq (A^{D,\dagger})^\dagger $, and hence, $ A^{D,\dagger} $ is not a partial isometry.
Moreover,
\begin{center}
$ A^{\dagger,D} = A^\dagger A A^D = \dfrac{1}{3} \begin{bmatrix}
3 & 0 & 0 \\
0 & 2 \sqrt{3} & 0 \\
0 & 0 & 0 \\
\end{bmatrix} $, and $ (A^{\dagger,D})^\dagger = \dfrac{1}{2} \begin{bmatrix}
2 & 0 & 0 \\
0 & \sqrt{3} & 0 \\
0 & 0 & 0 \\
\end{bmatrix} $.
\end{center}
Hence, $ (A^{\dagger,D})^\ast \neq (A^{\dagger,D})^\dagger $, and thus, $ A^{\dagger,D} $ is not a partial isometry.
Since ${\rm ind}(A) = 1 $, we can see, by Lemma \ref{CEP 1}$(iv)$, that $ A^{\scriptsize\textcircled{$\dagger$}} = A (A^2)^\dagger = \dfrac{1}{6} \begin{bmatrix}
6 & 0 & 0 \\
0 & 3 \sqrt{3} & 3 \\
0 & 3 & \sqrt{3} \\
\end{bmatrix},$ which is not a partial isometry.

 \item Let $ A = \begin{bmatrix}
0 & 1 & 0 & 0 \\
1 & 0 & 0 & 0 \\
0 & 0 & 0 & 1 \\
0 & 0 & 0 & 0 \\
\end{bmatrix} $. Obviously, $A$ is not a normal matrix. Also, we have that $ A^D = \begin{bmatrix}
0 & 1 & 0 & 0 \\
1 & 0 & 0 & 0 \\
0 & 0 & 0 & 0 \\
0 & 0 & 0 & 0 \\
\end{bmatrix},$ which is a partial isometry.
\end{enumerate}
\end{Ex}

 The next theorem shows when the Drazin inverse, the DMP inverse, and the MPD inverse of a partial isometry are also partial isometry.

 \begin{Th} \label{Drazin Inverse of P.I. matrix} For partial isometry $ A\in\mathbb{M}_n(\mathbb{C}) $, the following assertions
are true:
\begin{enumerate}[(i)]
\item $ A^D $ is partial isometry if and only if $ (A^\dagger)^D = (A^D)^\dagger $;
\item If $ A^D $ is a partial isometry, it holds that \\
$(a) \ A^{D,\dagger} $ is partial isometry iff $ (A^{D,\dagger})^\dagger = (A^\dagger)^{\dagger,D} $; \\
$(b) \ A^{\dagger,D} $ is partial isometry iff $ (A^{\dagger,D})^\dagger = (A^\dagger)^{D,\dagger} $.
\end{enumerate}

 \end{Th}

 \begin{proof}
To prove the assertion $(i)$, since $A$ is a partial isometry, it follows that
\begin{center}
$(A^\dagger)^D = (A^D)^\dagger \Leftrightarrow (A^\ast)^D = (A^D)^\dagger $.
\end{center}
So, the result in $(i)$ holds.

 To prove part $(a)$ in $ (ii) $, we assume that $A$ and $A^D$ are partial isometries. Then by Lemma \ref{DMP 1}($(i)$ and $(ii)$)
and using part $(i)$, we see that
\begin{eqnarray*}
(A^{D,\dagger})^\dagger = (A^{D,\dagger})^\ast & \Leftrightarrow & (A^{D,\dagger})^\dagger = (A^D A A^\dagger)^\ast \\
& \Leftrightarrow & (A^{D,\dagger})^\dagger = (A^D A A^\ast)^\ast \\
& \Leftrightarrow & (A^{D,\dagger})^\dagger = A A^\ast (A^D)^\ast \\
& \Leftrightarrow & (A^{D,\dagger})^\dagger = A A^\dagger (A^D)^\dagger \\
& \Leftrightarrow & (A^{D,\dagger})^\dagger = (A^\dagger)^\dagger (A^\dagger) (A^\dagger)^D \\
& \Leftrightarrow & (A^{D,\dagger})^\dagger = (A^\dagger)^{\dagger,D}.
\end{eqnarray*}
So, the assertion holds. By the same manner as in the above proof, part $(b)$ in $(ii)$ also holds; completing the proof.
\end{proof}
By Theorem \ref{Drazin Inverse of P.I. matrix}$(i)$, for partial isometries $A$, the Drazin inverse $A^D$ is a partial isometry if
and only if $ (A^\dagger)^D = (A^D)^\dagger $. Note that this equality also holds for EP matrices, but in general, it may fail for non-partial
isometries; see the following example.

\begin{Ex}\label{P.I 1}
Considering $ A = \begin{bmatrix}
0 & 1 & 1 \\
0 & 0 & 0 \\
0 & 0 & 1 \\
\end{bmatrix} $, we can see that $A$ is not partial isometry. Now, a simple computation shows that \\
$ A^D = \begin{bmatrix}
0 & 0 & 1 \\
0 & 0 & 0 \\
0 & 0 & 1 \\
\end{bmatrix} $, $ A^\dagger = \begin{bmatrix}
0 & 0 & 0 \\
1 & 0 & -1 \\
0 & 0 & 1 \\
\end{bmatrix} $,
$ (A^D)^\dagger = \dfrac{1}{2} \begin{bmatrix}
0 & 0 & 0 \\
0 & 0 & 0 \\
1 & 0 & 1 \\
\end{bmatrix} $, and $ (A^\dagger)^D = \begin{bmatrix}
0 & 0 & 0 \\
0 & 0 & -1 \\
0 & 0 & 1 \\
\end{bmatrix} $.
So, $ (A^D)^\dagger \neq (A^\dagger)^D $.
\end{Ex}

 The following result shows when $A^{\scriptsize\textcircled{$\dagger$}}$ is a partial isometry.

 \begin{Th}\label{A^k is P.I.}
For $ A\in\mathbb{M}_n(\mathbb{C})$ with $k = ind(A)$, if $A^k$ and $A^{k+1}$ are partial isometries, then
$A^{\scriptsize\textcircled{$\dagger$}}$ is partial isometry.
\end{Th}

 \begin{proof}
By Lemma \ref{CEP 1}($(vi)$ and $(iv)$), we see that
\begin{eqnarray*}
(A^{\scriptsize\textcircled{$\dagger$}})^\dagger & = & A^{k+1} (A^k)^\dagger \\
& = & ((A^{k+1})^\dagger)^\dagger (A^k)^\dagger \\
& = & ((A^{k+1})^\ast)^\dagger (A^k)^\dagger \\
& = & ((A^{k+1})^\dagger)^\ast (A^k)^\ast \\
& = & (A^k (A^{k+1})^\dagger)^\ast \\
& = & (A^{\scriptsize\textcircled{$\dagger$}})^\ast;
\end{eqnarray*}
as required.
\end{proof}

 The following example shows that the converse of Theorem \ref{A^k is P.I.} does not hold in general.

 \begin{Ex}
Consider the matrix $A$ as in Example \ref{P.I 1}. It is clear that ${\rm ind}(A) = 2$. By a simple computation, we have
$ A^{\scriptsize\textcircled{$\dagger$}} = \dfrac{1}{2} \begin{bmatrix}
1 & 0 & 1 \\
0 & 0 & 0 \\
1 & 0 & 1 \\
\end{bmatrix} $. Obviously, $ A^{\scriptsize\textcircled{$\dagger$}} $ is a partial isometry.
But
\begin{center}
$ A^2 = A^3 = \begin{bmatrix}
0 & 0 & 1 \\
0 & 0 & 0 \\
0 & 0 & 1 \\
\end{bmatrix} $, and $ (A^2)^\dagger = (A^3)^\dagger = \dfrac{1}{2} \begin{bmatrix}
0 & 0 & 0 \\
0 & 0 & 0 \\
1 & 0 & 1 \\
\end{bmatrix} $.
\end{center}
Hence, $ A^2$ and $ A^3$ are not partial isometries.
\end{Ex}

\begin{Th}
For CEPD matrix $ A\in\mathbb{M}_n(\mathbb{C})$ with $k= ind(A)$ and the core part $A_1$, if
$A^k$ and $A^{k+1}$ are partial isometries, then the following assertions are true:
\begin{enumerate}[(i)]
\item $ A^D (A^{\scriptsize\textcircled{$\dagger$}})^\ast A^{\scriptsize\textcircled{$\dagger$}} =
(A^{\scriptsize\textcircled{$\dagger$}})^2 A^{k+1} (A^k)^\dagger$;
\item $ (A^{\scriptsize\textcircled{$\dagger$}})^\ast A^k = A^D A^k = A_1 A^k$.
\end{enumerate}
\end{Th}

 \begin{proof}
To prove the assertion $(i)$, by Lemma \ref{CEP 1}$(vii)$, Theorem \ref{A^k is P.I.} and using the fact that EP
partial isometries are normal, we conclude that $A^{\scriptsize\textcircled{$\dagger$}}$ is a normal matrix. Hence, by Theorem
\ref{CEPD equivalent}$(vi)$ and Lemma \ref{CEP 1}$(vi)$, we have
\begin{eqnarray*}
A^D (A^{\scriptsize\textcircled{$\dagger$}})^\ast A^{\scriptsize\textcircled{$\dagger$}} & = & A^D A^{\scriptsize\textcircled{$\dagger$}} (A^{\scriptsize\textcircled{$\dagger$}})^\ast \\
& = & (A^{\scriptsize\textcircled{$\dagger$}})^2 (A^{\scriptsize\textcircled{$\dagger$}})^\ast \\
& = & (A^{\scriptsize\textcircled{$\dagger$}})^2 A^{k+1} (A^k)^\dagger.
\end{eqnarray*}
So, the result in $(i)$ holds. To prove the first equality in assertion $(ii)$, by Lemma \ref{CEP 1}$(vii)$, using
Theorems \ref{A^k is P.I.}, \ref{Drazin- Star} and \ref{CEPD equivalent}$(vi)$ and relation (\ref{D}), we have
\begin{eqnarray*}
(A^{\scriptsize\textcircled{$\dagger$}})^\ast A^k & = & (A^{\scriptsize\textcircled{$\dagger$}})^D A^k \\
& = & (A^{\scriptsize\textcircled{$\dagger$}})^D (A^{\scriptsize\textcircled{$\dagger$}})^\ast A^{\scriptsize\textcircled{$\dagger$}} A^k \\
& = & ((A^{\scriptsize\textcircled{$\dagger$}})^D)^2 A^D A^k \\
& = & ((A^D)^D)^2 A^D A^k \\
& = & A^D A^k;
\end{eqnarray*}
as required. To prove the second equality in this assertion, by Theorem \ref{CEPD equivalent}$(vi)$ and Lemma \ref{DMP 1}$(vi)$,
obviously, $(A^{\scriptsize\textcircled{$\dagger$}})^D A^k = A_1 A^k$.
So, the result in $(ii)$ also holds; completing the proof.
\end{proof}

 In the following theorem, we find some conditions for which a partial isometry is CEPD.

 \begin{Th}\label{P.I. and CEPD}
For partial isometry $ A\in\mathbb{M}_n(\mathbb{C})$ with $k = ind(A)$, if
any of both cases:
$ (a) \ A^k $ is Hermitian and a partial isometry, or
$ (b) \ A^{k+1} $ is Hermitian and a partial isometry, holds, then $A$ is CEPD.
\end{Th}

\begin{proof}
At first, we assume that $A^k$ is a partial isometry and Hermitian. Hence, $ (A^k)^\dagger = A^k$.
So, by Lemma \ref{CEP 1}$(iii)$ and the fact that $ A^D A^k = A^k A^D $, we can see that
\begin{eqnarray*}
A^{\scriptsize\textcircled{$\dagger$}} A^D & = & A^D A^k (A^k)^\dagger A^D \\
& = & A^D A^k A^k A^D \\
& = & A^D A^k A^D A^k \\
& = & A^D (A^D A^k (A^k)^\dagger) \\
& = & A^D A^{\scriptsize\textcircled{$\dagger$}};
\end{eqnarray*}
as required. Secondly, we assume that $A^{k+1}$ is a partial isometry and Hermitian. Hence,
$ (A^{k+1})^\dagger = A^{k+1} $. Thus, by Lemma \ref{CEP 1}$(iv)$ and the fact that $ A^D A^k = A^k A^D $, we have
\begin{eqnarray*}
A^D A^{\scriptsize\textcircled{$\dagger$}} & = & A^D A^k (A^{k+1})^\dagger \\
& = & A^k A^D (A^{k+1})^\dagger \\
& = & A^k A^D A^{k+1} \\
& = & A^k A^{k+1} A^D \\
& = & A^k (A^{k+1})^\dagger A^D \\
& = & A^{\scriptsize\textcircled{$\dagger$}} A^D.
\end{eqnarray*}
Thus, the proof is complete.
\end{proof}

 Note that if $A$ is a CEPD matrix, then $A^k$ and $A^{k+1}$ may not be partial isometries, where $k={\rm ind}(A)$.
Hence, in view of the following example, the converse of Theorem \ref{P.I. and CEPD} does not hold in general.

\begin{Ex}
Let $ A = \begin{bmatrix}
1 & 2 & 0 & 0 \\
2 & 1 & 0 & 0 \\
0 & 0 & 0 & 1 \\
0 & 0 & 0 & 0 \\
\end{bmatrix} $. Obviously, ${\rm ind}(A) = 2 $ and $A$ is $2-$EP. Hence,
by Theorem \ref{CEP and D}, $A$ is CEPD. By a simple computation, we have
\begin{center}
$ A^2 = \begin{bmatrix}
5 & 4 & 0 & 0 \\
4 & 5 & 0 & 0 \\
0 & 0 & 0 & 0 \\
0 & 0 & 0 & 0 \\
\end{bmatrix} $ and $ (A^2)^\dagger = \dfrac{1}{9} \begin{bmatrix}
5 & -4 & 0 & 0 \\
-4 & 5 & 0 & 0 \\
0 & 0 & 0 & 0 \\
0 & 0 & 0 & 0 \\
\end{bmatrix} $.
\end{center}
Thus, $A^2$ is not a partial isometry. Moreover,
\begin{center}
$ A^3 = \begin{bmatrix}
13 & 14 & 0 & 0 \\
14 & 13 & 0 & 0 \\
0 & 0 & 0 & 0 \\
0 & 0 & 0 & 0 \\
\end{bmatrix} $ and $ (A^3)^\dagger = \dfrac{1}{27} \begin{bmatrix}
-13 & 14 & 0 & 0 \\
14 & -13 & 0 & 0 \\
0 & 0 & 0 & 0 \\
0 & 0 & 0 & 0 \\
\end{bmatrix} $.
\end{center}
Thus, $A^3$ is not a partial isometry.
\end{Ex}

 The final example shows that if $A$ is a partial isometry with $k= {\rm ind}(A)$, then $A^k$ and $A^{k+1}$ may not be partial isometries.
Also, it shows that a partial isometry may not be CEPD. Note that by Theorem \ref{normal is CEPD}, normal matrices are CEPD, and so, obviously,
we can find CEPD matrices which are not partially isometric.

 \begin{Ex}
Let $A = \dfrac{1}{2} \begin{bmatrix}
0 & 0 & 0 & 0 & 0 & 0 \\
\sqrt{2} & -\sqrt{2} & 0 & 0 & 0 & 0 \\
1 & 1 & 0 & 0 & 0 & 0 \\
1 & 1 & 0 & 0 & 0 & 0 \\
0 & 0 & 0 & 0 & 0 & \sqrt{2} \\
0 & 0 & 0 & 0 & 0 & \sqrt{2} \\
\end{bmatrix} $. Obviously, $A$ is a partial isometry and ${\rm ind}(A) = 2 $. By a simple computation, we have
\begin{center}
$(A^2)^\ast = \dfrac{1}{4} \begin{bmatrix}
0 & -2 & \sqrt{2} & \sqrt{2} & 0 & 0 \\
0 & 2 & -\sqrt{2} & -\sqrt{2} & 0 & 0 \\
0 & 0 & 0 & 0 & 0 & 0 \\
0 & 0 & 0 & 0 & 0 & 0 \\
0 & 0 & 0 & 0 & 0 & 0 \\
0 & 0 & 0 & 0 & 2 & 2 \\
\end{bmatrix} $, and
\end{center}
\begin{center}
$(A^2)^\dagger = \dfrac{1}{4} \begin{bmatrix}
0 & -2 & \sqrt{2} & \sqrt{2} & 0 & 0 \\
0 & 2 & -\sqrt{2} & -\sqrt{2} & 0 & 0 \\
0 & 0 & 0 & 0 & 0 & 0 \\
0 & 0 & 0 & 0 & 0 & 0 \\
0 & 0 & 0 & 0 & 0 & 0 \\
0 & 0 & 0 & 0 & 4 & 4 \\
\end{bmatrix} $.
\end{center}
Thus $(A^2)^\ast \neq (A^2)^\dagger $, and so, $A^2$ is not a partial isometry. Moreover, we have
\begin{center}
$(A^3)^\ast = \dfrac{1}{4} \begin{bmatrix}
0 & \sqrt{2} & -1 & -1 & 0 & 0 \\
0 & -\sqrt{2} & 1 & 1 & 0 & 0 \\
0 & 0 & 0 & 0 & 0 & 0 \\
0 & 0 & 0 & 0 & 0 & 0 \\
0 & 0 & 0 & 0 & 0 & 0 \\
0 & 0 & 0 & 0 & \sqrt{2} & \sqrt{2} \\
\end{bmatrix} $, and
\end{center}
\begin{center}
$(A^3)^\dagger = \dfrac{1}{2}\begin{bmatrix}
0 & \sqrt{2} & -1 & -1 & 0 & 0 \\
0 & -\sqrt{2} & 1 & 1 & 0 & 0 \\
0 & 0 & 0 & 0 & 0 & 0 \\
0 & 0 & 0 & 0 & 0 & 0 \\
0 & 0 & 0 & 0 & 0 & 0 \\
0 & 0 & 0 & 0 & 2 \sqrt{2} & 2 \sqrt{2} \\
\end{bmatrix} $.
\end{center}
Hence $(A^3)^\ast \neq (A^3)^\dagger $, and so, $A^3$ is not a partial isometry. Also, we have
\begin{center}
$ A^D = \begin{bmatrix}
0 & 0 & 0 & 0 & 0 & 0 \\
\sqrt{2} & -\sqrt{2} & 0 & 0 & 0 & 0 \\
-1 & 1 & 0 & 0 & 0 & 0 \\
-1 & 1 & 0 & 0 & 0 & 0 \\
0 & 0 & 0 & 0 & 0 & \sqrt{2} \\
0 & 0 & 0 & 0 & 0 & \sqrt{2} \\
\end{bmatrix} $,
\end{center}
\begin{center}
$ A^{\scriptsize\textcircled{$\dagger$}} =\dfrac{1}{4} \begin{bmatrix}
0 & 0 & 0 & 0 & 0 & 0 \\
0 & -2 \sqrt{2} & 2 & 2 & 0 & 0 \\
0 & 2 & -\sqrt{2} & -\sqrt{2} & 0 & 0 \\
0 & 2 & -\sqrt{2} & -\sqrt{2} & 0 & 0 \\
0 & 0 & 0 & 0 & 2 \sqrt{2} & 2 \sqrt{2} \\
0 & 0 & 0 & 0 & 2 \sqrt{2} & 2 \sqrt{2} \\
\end{bmatrix}.$
\end{center}
Therefore, $ A^D \neq A^{\scriptsize\textcircled{$\dagger$}} $, and so, by Theorem \ref{CEPD equivalent}$(vi)$, $A$ is not CEPD.
\end{Ex}


\noindent
\section{Some applications}

In this section, we are going to state some applications of the CEPD matrices and partial isometries in solving systems of linear equations.

\begin{Prop}
Let $ A\in\mathbb{M}_n(\mathbb{C})$ be a CEPD matrix and $b\in\mathbb{M}_{n\times 1}(\mathbb{C})$. Then the linear system
\begin{center}
$A(x-A^{\scriptsize\textcircled{$\dagger$}} b) = 0$
\end{center}
is consistent and its general solution is
\begin{center}
$x=A^{c \dagger} b + (I-A^\dagger A) y$,
\end{center}
where $y\in \mathbb{M}_{n\times 1}(\mathbb{C})$ is arbitrary.
\end{Prop}

\begin{proof}
At first we show that $A^{c \dagger} = A^\dagger A A^{\scriptsize\textcircled{$\dagger$}}$. For this,
since $A$ is CEPD, it follows, by Theorem \ref{CEPD equivalent}($(i)\Leftrightarrow(vi)$), that
$A^D = A^{\scriptsize\textcircled{$\dagger$}}$. We know, by Lemma \ref{CEP 1}$(ii)$ and (\ref{MP}), that $A A^{\scriptsize\textcircled{$\dagger$}}$
and $A A^\dagger$ are Hermitian matrices. So, by Lemma \ref{DMP 1}$(iii)$, we have
\begin{eqnarray*}
A^{c \dagger} & = & A^\dagger A A^D A A^\dagger \\
& = & A^\dagger A A^{\scriptsize\textcircled{$\dagger$}} A A^\dagger \\
& = & A^\dagger (A A^{\scriptsize\textcircled{$\dagger$}})^\ast (A A^\dagger)^\ast \\
& = & A^\dagger (A A^\dagger A A^{\scriptsize\textcircled{$\dagger$}})^\ast \\
& = & A^\dagger (A A^{\scriptsize\textcircled{$\dagger$}})^\ast \\
& = & A^\dagger A A^{\scriptsize\textcircled{$\dagger$}}.
\end{eqnarray*}

Now, the result follows from Lemma \ref{CEP 1}($ix$) and \cite[Corollary 9.1]{MKS}.
\end{proof}

 In numerical analysis, in order to solve a linear system $Ax=b$, it is usually considered a preconditioner to
accelerate iterative methods. That is, nonsingular matrices $P$ such that $P^{-1}A$ is close to the identity matrix and so the alternative
linear system $P^{-1}(Ax-b)=0$ is solved (provided that it has the same solution as the original one). When $A$ is singular, the generalized
inverses can be considered as a starting point to do this job. In what follows, we consider the situation $A^D(Ax-b)=0$.

\begin{Prop}
Let $ A\in\mathbb{M}_n(\mathbb{C})$ be a CEPD matrix and $b\in\mathbb{M}_{n\times 1}(\mathbb{C})$. Then the linear system
\begin{center}
$A^D (A x - b) = 0$
\end{center}
is consistent and its general solution is
\begin{center}
$x = A^D b + (I - A^D A) y$,
\end{center}
where $y\in\mathbb{M}_{n\times 1}(\mathbb{C})$ is arbitrary.
\end{Prop}
\begin{proof}
By \cite[Corollary 4.1]{MSM} and Theorem \ref{CEPD equivalent}($(i)\Leftrightarrow(vi)$), the result holds.
\end{proof}

\begin{Prop}
Let $ A\in\mathbb{M}_n(\mathbb{C})$ be a CEPD matrix and $b\in\mathbb{M}_{n\times 1}(\mathbb{C})$. Then
the linear system
\begin{center}
$A x = b $
\end{center}
has a unique solution in $\mathcal{R}(A^{\scriptsize\textcircled{$\dagger$}})$ given by
\begin{center}
$ x = A^{\scriptsize\textcircled{$\dagger$}} b$.
\end{center}
\end{Prop}

\begin{proof}
Let ${\rm ind}(A) = k$. By (\ref{CEP}), we see that $\mathcal{R}(A^k) = \mathcal{R}(A^{\scriptsize\textcircled{$\dagger$}})$.
Now, the result follows from \cite[p. 167]{BG} and Theorem \ref{CEPD equivalent}($(i)\Leftrightarrow(vi)$).
\end{proof}

To see some applications of partial isometries in solving systems of linear equations, consider the next two propositions.

\begin{Prop}
Let $ A\in\mathbb{M}_n(\mathbb{C})$ be a partial isometry and $b\in \mathcal{R}(A)$. Then
\begin{center}
$x = A^\ast b$
\end{center}
is a solution of the linear system
\begin{center}
$A x = b$.
\end{center}
\end{Prop}
\begin{proof}
By \cite[Theorem 2.3.1($(i)\Leftrightarrow(ii)$)]{MBM} and the fact that $A^\dagger = A^\ast$, the result holds.
\end{proof}

For partial isometry $ A\in\mathbb{M}_n(\mathbb{C})$, it is known, by \cite[Def. 2.1]{D}, that $A^{D,\ast} $ is an outer inverse of $A$. So,
by \cite[Theorem 4.2]{D}, we have the following result.

\begin{Prop}
Let $ A\in\mathbb{M}_n(\mathbb{C})$ be a partial isometry with $ind(A)= k$, and $b\in \mathcal{R}(A A^D)$.
Then the linear system
\begin{center}
$A x = b $
\end{center}
in the space $\mathcal{R}(A^k)$ has a unique solution given by
\begin{center}
$x = A^{D,\ast} b.$
\end{center}
\end{Prop}


\noindent
\section{Conclusion}

 According to the important applications of the Drazin inverse and the core EP inverse in various areas of applied sciences,
in the present paper, the class of matrices $A$ for which the Drazin inverse $A^D$ is equal to the core EP inverse $A^{\scriptsize\textcircled{$\dagger$}}$, entitled CEPD matrices, has been considered.
Also, several algebraic properties of these matrices have been investigated. Moreover, partial isometries and also the Drazin inverse, the core EP inverse, the MPD inverse, the DMP inverse of these matrices have been studied. In addition,
some results for CEPD matrices when their core EP inverse is a partial isometry have been obtained. Some conditions for which
partial isometries are CEPD have also been given.
Moreover, some applications of CEPD matrices and partial isometries in solving systems of linear equations have been presented.
Also, to illustrate the main results, some numerical examples have been provided.


\vspace{0.5cm} \noindent

 {\bf Acknowledgements}

 The third author is supported by the Ministry of Education, Science and Technological Development,
Republic of Serbia [grant number 451-03-65/2024-03/200124].
The fourth author was partially supported by Ministerio de Ciencia e Innovaci\'on of Spain [grant RED2022-134176-T],
by Universidad Nacional de La Pampa, Facultad de Ingenier\'{\i}a, Argentina [grant number Resol. Nro. 172/2024],
by Universidad Nacional de R\'{\i}o Cuarto, Argentina [grant number 449/2024] and also by Universidad Nacional del Sur,
Argentina [grant number 24/ZL22]. The authors are very grateful to the anonymous referees for helpful comments and useful suggestions.


\end{document}